\documentclass[a4paper]{amsart}
\usepackage{amsmath,amstext,amssymb,amsfonts,amscd,amsthm}
\usepackage[numbers,sort&compress]{natbib}
\usepackage{mathdots,mathrsfs,enumerate}
\usepackage{extarrows}
\usepackage{mathrsfs}  
\usepackage{dsfont}  
\usepackage{graphicx,color}
\usepackage{tikz}
\usepackage{stfloats}
\usetikzlibrary{3d,calc,patterns,graphs,arrows}
\usepackage{tikz-cd,array,diagbox}
\usepackage{changepage}
\usepackage{geometry}
\usepackage[CJKbookmarks=true,unicode,colorlinks,linkcolor=blue,anchorcolor=blue,citecolor=blue]{hyperref}
\usepackage{subfigure}
\numberwithin{equation}{section}

\usepackage{xcolor}
\definecolor{BeanPasteGreen}{RGB}{200,232,200}
\definecolor{WaterBlue}{RGB}{185,220,237}

\newtheorem{lemma}{Lemma}[section]
\newtheorem{theorem}{Theorem}[section]

\newtheorem{proposition}{Proposition}[section]

\newtheorem{property}{Property}[section]

\newtheorem{remark}{Remark}[section]

\newenvironment{proof*}
\numberwithin{equation}{section}

\newcommand{\op}[1]{\operatorname{#1}} 

\newcommand{\ds}{\displaystyle}

\newcommand{\p}{\partial}

\renewcommand{\t}{\triangle}

\newcommand{\ls}{\lesssim}

\newcommand{\sm}{\setminus}

\newcommand{\abs}[1]{\left\vert#1\right\vert}
\newcommand{\norm}[1]{\left\Vert#1\right\Vert}
\newcommand{\set}[1]{\left\{#1\right\}}

\newcommand{\frb}[1]{\left(#1\right)}

\newcommand{\ol}{\overline}

\newcommand{\wt}{\widetilde}

\newcommand{\Gra}{\Longrightarrow}

\newcommand{\Z}{\mathbb Z}

\newcommand{\R}{\mathbb R}

\renewcommand{\a}{\alpha}
\renewcommand{\b}{\beta}
\newcommand{\g}{\gamma}
\providecommand{\G}{}
\renewcommand{\G}{\Gamma}
\renewcommand{\d}{\delta}

\newcommand{\e}{\epsilon}
\newcommand{\ve}{\varepsilon}

\renewcommand{\th}{\theta}

\renewcommand{\k}{\kappa}

\renewcommand{\l}{\lambda}

\newcommand{\s}{\sigma}

\newcommand{\vp}{\varphi}

\newcommand{\Om}{\Omega}

\allowdisplaybreaks[3]

\begin{document}

\title[Interior Hessian estimates for Hessian quotient equations]
{Interior Hessian estimates for $k$-convex solutions of Hessian quotient equations}

\author{Zhisu Li}
\address{School of Mathematics and Center for Nonlinear Studies,
Northwest University, Xi'an, 710127, People's Republic of China}
\email{lizhisu@nwu.edu.cn}

\author{Ke Wu}
\address{School of Mathematics and Center for Nonlinear Studies,
Northwest University, Xi'an, 710127, People's Republic of China}
\email{wukemath@stumail.nwu.edu.cn}

\date{October 11, 2026.}

\begin{abstract}
We study interior Hessian estimates
for $k$-convex solutions of Hessian quotient equations $\frac{\s_k(D^2u)}{\s_\ell(D^2u)}=1$
with $2\leq k\leq n$, $0\leq\ell<k$, and $k-\ell\in\{1,2\}$.
We first adopt the unified quantitative separation-propagation method of Li--Wu \cite{LW26b}
to derive a doubling inequality,
and then use the compactness argument of Shankar--Yuan \cite{SY25}
to obtain the estimates.
This approach also applies to a class of general Hessian equations.
\end{abstract}

\maketitle
\tableofcontents

Keywords:
Hessian quotient equations,
interior Hessian estimates,
$k$-convexity,
separation-propagation method,
doubling inequalities

2020 Mathematics Subject Classification:
Primary:
35B45; 
Secondary:
35B65, 
35J60.  

\section{Introduction}\label{sec.introduction}

In this paper, we study interior Hessian estimates for the Hessian quotient equations
\begin{equation}\label{eqn.sigma-quotient}
F(D^2u):=\frac{\s_k(D^2u)}{\s_\ell(D^2u)}=1,
\end{equation}
where $2\leq k\leq n$, $0\leq\ell<k$, and $k-\ell\in\{1,2\}$.
Here $\s_j(D^2u):=\s_j(\l(D^2u))$ denotes the $j$-th
elementary symmetric function of the eigenvalue-vector $\l(D^2u)$ of $D^2u$,
with $\s_0(D^2u):=1$.

The natural elliptic branch for the $k$-Hessian operator is the G\aa rding cone
\[
\G_k:=\set{\l\in\R^n:\s_j(\l)>0\ \hbox{for all}\ 1\leq j\leq k}.
\]
The Hessian quotient operator $F$ is elliptic on $\G_k$.
In this setting, a function $u\in C^2(\Om)$ is called $k$-convex
if $\l(D^2u(x))\in\G_k$ for any $x\in\Om$.

\begin{theorem}\label{thm.Hessian-estimate}
Let $2\leq k\leq n$, $0\leq\ell<k$, and $k-\ell\in\{1,2\}$.
Suppose that $u\in C^\infty(B_1)$ is a $k$-convex solution of \eqref{eqn.sigma-quotient} in $B_1\subset\R^n$.
Then
\[
|D^2u(0)|\leq C\frb{n,k,\|u\|_{C^1(B_1)}}.
\]

From the interior gradient estimate for Hessian quotient equations by Chen \cite{C15},
we further obtain for $k$-convex solutions of \eqref{eqn.sigma-quotient} in $B_2$ that
\[
|D^2u(0)|
\leq C\frb{n,k,\norm{u}_{L^\infty(B_2)}}.
\]
\end{theorem}

\begin{remark}
The restriction $k-\ell\in\{1,2\}$ is sharp
for interior Hessian estimates controlled by the $C^1$ norm,
even for smooth convex solutions with constant right-hand side.
For $\ell=0$ and $k\geq3$,
this follows from the counterexamples of Pogorelov \cite{P78} and Urbas \cite{U90}.
For $\ell\geq1$ and $k-\ell\geq3$,
such counterexamples can be found in \cite[Section 7]{Lu25}.
\end{remark}

\begin{remark}
Theorem \ref{thm.Hessian-estimate} extends to a class of general Hessian equations
by the same argument.
See Section \ref{sec.general-Hessian} for details.
\end{remark}

\begin{remark}
The second author \cite[Theorem  1.1]{W26} established Theorem \ref{thm.Hessian-estimate}
for $1\leq\ell<k<n$, $k-\ell\in\set{1,2}$ and $2k>n$
by combining a concavity inequality and a pointwise doubling argument
with the compactness argument of Shankar--Yuan \cite{SY25}.
\end{remark}

\begin{remark}
Qiu--Yan \cite[Theorem  1.1]{QY26b} recently
established the first estimate in Theorem \ref{thm.Hessian-estimate}
by adapting the quantitative separation-propagation method
of the authors \cite{LW26b}.
\end{remark}

The study of interior Hessian estimates for Hessian equations has seen many advances.

For the $k$-Hessian equation $\s_k(D^2u)=f$,
Heinz \cite{H59} proved the classical estimate for the two-dimensional
Monge--Amp\`ere equation; see also \cite{CHO16,Liu21}.
Pogorelov \cite{P78} constructed singular Monge--Amp\`ere solutions in dimensions $n\geq3$,
and Urbas \cite{U90} obtained analogous examples
for $k$-Hessian equations with $k\geq3$.
For the quadratic equation $\s_2(D^2u)=1$,
Warren--Yuan \cite{WY09} established interior Hessian estimates in dimension three,
and Qiu \cite{Q24} treated positive variable right-hand sides.
The four-dimensional Hessian estimate for $\s_2(D^2u)=1$
was subsequently established by Shankar--Yuan \cite{SY25}.
They also obtained estimates in higher dimensions under dynamic semiconvexity,
while Fan \cite{Fan26} extended these results to positive variable right-hand sides.
Earlier estimates in arbitrary dimensions were obtained
under almost convexity \cite{MSY19}, semiconvexity \cite{SY20},
or the condition $\s_3(D^2u)\geq-A$ \cite{GQ19}.
Recently, the authors \cite{LW26b} established interior Hessian estimates
for $\s_2(D^2u)=1$ in arbitrary dimensions under $2$-convexity alone.

For Hessian quotient equations,
the interior estimates for $\s_3/\s_1=1$ in dimensions three and four
follow from the special Lagrangian structure \cite{CWY09,WY14,Z24},
while Lu \cite{Lu23} gave a proof based on a Jacobi inequality in dimension three.
Lu \cite{Lu25} established Hessian estimates for $\s_n/\s_\ell=f(x)$ with $n-\ell\in\set{1,2}$,
and constructed counterexamples for convex solutions when $k-\ell\geq3$.
Jiao--Sui \cite{JS26} obtained estimates for $\s_2/\s_1$
under $2$-convexity in dimension three and semiconvexity in higher dimensions,
while Mei--Yan \cite{MY26} treated semiconvex solutions
of $\s_3/\s_\ell=1$ for $\ell=1,2$ in arbitrary dimensions.
The authors \cite{LW26b} also treated $\s_2/\s_1=1$
by reducing it to the quadratic Hessian equation through a quadratic shift.
For general quotients with $k-\ell\in\{1,2\}$,
Lu--Tsai \cite{LT26a} obtained interior Hessian estimates for convex solutions
under a structural concavity assumption.
This assumption was removed for adjacent quotients by Tsai \cite{T26}
and for both cases by the authors \cite{LW26a}.
Dong--Zhang \cite{DZ26} extended these estimates
to admissible semiconvex solutions with positive $C^2$ right-hand sides.
For constant right-hand sides,
the second author \cite{W26} obtained estimates under $k$-convexity when $2k>n$.
Around the same time,
Qiu--Yan \cite{QY26b} established interior Hessian estimates
for the equations considered here
under $k$-convexity in arbitrary dimensions.

\medskip

Three related approaches to interior Hessian estimates
provide the background for our proof.

The first is the integral method of Warren--Yuan \cite{WY09}.
It combines a Jacobi inequality with a mean-value inequality on the gradient graph
to reduce the pointwise Hessian bound
to a weighted integral,
which is then controlled by integration by parts.
Qiu \cite{Q24} extended this method to variable right-hand sides
and refined the argument to avoid the Sobolev inequality
by incorporating a maximum-principle argument.
For semiconvex solutions of the $\s_2$ equation,
Shankar--Yuan \cite{SY20} used a quadratic shift
and a Legendre--Lewy transform,
building on Chang--Yuan's observation \cite{CY10}
that the transformed linearized operator is uniformly elliptic
after normalization.
Their transformation rule transfers the subharmonicity
provided by the Jacobi inequality to the new variables,
where a local maximum principle yields the mean-value estimate.
Lu \cite{Lu23,Lu25} developed this approach for Hessian quotient equations.
In \cite{Lu25}, he used the concavity inequalities of Guan--Sroka \cite{GS26}
to control the third-order terms
and derive a Jacobi inequality for $\log\l_{\max}(D^2u)$.
A Legendre transform then gives the mean-value estimate,
while an induction argument controls the integrals
generated by repeated integration by parts.
This approach was further developed
for convex solutions in \cite{LT26a,T26,LW26a}
and admissible semiconvex solutions in \cite{DZ26}.

The second is the compactness method of Shankar--Yuan \cite{SY25}.
It combines an almost Jacobi inequality,
a doubling inequality extending Qiu's method \cite{Q24},
an Alexandrov-type differentiability result adapted from \cite{CT05},
and Savin's small perturbation theorem \cite{Sav07}.
Shankar \cite{S26} applied this strategy to the special Lagrangian equation
by combining the radial-derivative auxiliary function $x\cdot Du-u$
of Guan--Qiu \cite{GQ19}
with Korevaar-type exponential cutoffs \cite{K87}.
Fung \cite{F26a} adapted this construction to convex solutions
of Hessian quotient equations with gradient-dependent right-hand sides
under suitable structural assumptions.
The second author \cite{W26} derived a Jacobi inequality
from a concavity inequality for Hessian quotients,
and used barriers adapted from \cite{LW26b}
to establish the doubling inequality required by this method.

\vspace{2.5pt}

The third is the quantitative separation-propagation method of the authors \cite{LW26b}.
It first constructs a comparison solution $v$
of a related nonlinear equation
with boundary data $u$,
then establishes a uniform positive lower bound for $v-u$ on a small ball,
propagates this bound to a larger ball,
and finally applies the Chou--Wang Pogorelov estimate \cite{CW01}
on a suitable superlevel set of $v-u$.
It is worth mentioning that
Mooney \cite{M21} used Chou--Wang's estimate to obtain interior Hessian estimates
for convex viscosity solutions of $\s_2(D^2u)=1$,
and later Mooney \cite{M25} also emphasized its potential for the regularity of solutions.
Nevertheless,
prior to \cite{LW26b},
in view of the counterexamples
to interior Hessian estimates for $k$-Hessian equations with $k\geq3$
and for Hessian quotient equations with $k-\ell\geq3$,
it remained unclear whether such estimates
for the quadratic Hessian equation and quotients with $k-\ell\in\{1,2\}$
held in higher dimensions under admissibility alone,
without additional convexity or semiconvexity.
In particular, for $\s_2(D^2u)=1$ in dimensions $n\geq5$,
the search for singular viscosity solutions and the study of finer partial regularity
were also viewed as possible research directions.
In this context,
the authors \cite{LW26b} used this separation-propagation method
to give the first affirmative answer regarding interior Hessian estimates
for $2$-Hessian (quotient) equations $\s_2(D^2u)=\frac{\s_2}{\s_0}(D^2u)=1$ and $\frac{\s_2}{\s_1}(D^2u)=1$
in arbitrary dimensions under natural $2$-convexity alone,
and this method was soon adapted
to obtain interior estimates for several related equations.
Specifically,
Chen--Zhou--Zhu \cite{CZZ26} considered the $2$-Hessian equation with positive $C^\a$ right-hand sides;
Qiu--Yan \cite{QY26a} studied the constant graphical scalar curvature equation
by combining a distance Jacobi inequality with a two-surface maximum-principle argument;
and Han--Ma \cite{HM26} further treated
graphical curvature quotients $\s_k(\k)/\s_{k-2}(\k)=1$ with $3\leq k<n$
by deriving a quantitative compression inequality.
For the adjacent Hessian quotient $F=\s_k/\s_{k-1}$,
Qiu--Yan \cite{QY26b} followed the same strategy with a different comparison solution.
They solve
\[
\s_k(D^2v)-\s_{k-1}(D^2v)=-\ve
\ \text{in}\ B_{7/8},
\quad
v=u
\ \text{on}\ \p B_{7/8}.
\]
The Newton-tensor identities, an interior gradient bound,
and Qiu's doubling argument \cite{Q24}
give $v-u\geq\d_0>0$ in $B_{1/2}$.
The concavity of $F$ gives a directional Jacobi-type inequality
for upper test functions touching $\log\l_{\max}$.
Together with the exponential factor involving $|Du|^2$
and the bounds
\[
\sum_{i,j=1}^n\frac{\p F}{\p u_{ij}}(D^2u)(v-u)_{ij}\geq-\frac{k\ve}{\s_{k-1}(D^2u)},
\quad
\s_{k-1}(D^2u)\geq c(n,k)\l_{\max},
\]
they obtain the Pogorelov estimate
\[
(v-u-\d_0/2)^4\l_{\max}\leq C\frb{n,k,\|u\|_{C^1(B_1)}}
\quad\text{in}\
\set{x\in B_{7/8}:v-u>\d_0/2}.
\]
The uniform separation on $B_{1/2}$ then gives the interior Hessian bound.

\medskip

In this paper, we give an alternative proof
by deriving a doubling inequality for $\l_{\max}$ through quantitative separation-propagation
and then applying the compactness argument in \cite{SY25,W26},
based on Alexandrov-type differentiability \cite{CT05,W26}
and Savin's small perturbation theorem \cite{Sav07}.
As observed in \cite{LW26b},
separation propagation can be viewed as doubling for $(v-u)^{-1}$.
Here the propagated separation enables us to use $v-u-\d$
as a cutoff factor in an auxiliary function
of the form employed in Pogorelov estimates;
the corresponding maximum-principle calculation
then gives a new proof of the doubling inequality,
without establishing a Jacobi inequality for $\log\l_{\max}(D^2u)$.
To this end, we begin by solving the following Dirichlet problem for the original quotient operator,
with only the right-hand side modified:
\[
F(D^2v)=(1-\zeta/4)^{k-\ell}\ \text{in}\ B_{3/2},
\quad
v=u\ \text{on}\ \p B_{3/2},
\]
where $\zeta$ is a smooth cutoff function supported in $B_{1/4}$
with $0\leq\zeta\leq1$ and $\zeta=1$ in $B_\rho$ for $\rho:=1/32$.
The existence of $v$ follows from \cite{CNS85,Tru95},
and the comparison principle gives
\[
v-u\geq0\quad\text{in}\ B_{3/2}.
\]

Next, we obtain a quantitative lower bound for $v-u$.
To see this, consider
\[
\wt F^{ij}
:=\s_\ell(D^2v)\frac{\p F}{\p v_{ij}}(D^2v)
=\frac{\p\s_k}{\p v_{ij}}(D^2v)-F\frac{\p\s_\ell}{\p v_{ij}}(D^2v).
\]
In $B_\rho$, the matrix $(\wt F^{ij})$ is a constant linear combination of divergence-free Newton tensors,
and its trace is bounded by $C\s_{k-1}(D^2v)$.
The concavity of $F^{1/(k-\ell)}$ gives
\[
L_v(v-u)\leq-c(n,k)<0,
\]
where $L_v:=\frac{\p F}{\p v_{ij}}(D^2v)\p_{ij}$ denotes the linearized operator of $F$ at $v$.
Multiplying by $\s_\ell(D^2v)$ and integrating by parts twice against a cutoff give
\[
1\ls\int_{B_1}(v-u)\s_{k-1}(D^2v).
\]
The Newton-tensor identities and the gradient estimate then yield the seed separation
\[
v-u\geq\ve_0>0
\quad\text{on a small ball contained in}\ B_{1/2}.
\]
To propagate this gap,
for each $y\in B_{1/2}$,
we introduce the auxiliary functions as in \cite{LW26b},
\begin{gather*}
\vp_y(x)
:=(x-y)\cdot Dv(x)-v(x)+v(y)
+\frac{\a}{2}|x-y|^2-\b|Dv(x)|^2,
\\
\psi_y
:=e^{(c_y-\vp_y)/\g}-1,
\end{gather*}
where $\vp_y$ is adapted from the test functions in \cite{S26,GQ19,Q24}
and $\psi_y$ from the Korevaar-type exponential cutoffs in \cite{K87,S26}.
Here the parameters are chosen so that the component $\Om_y$
of $\set{\vp_y<c_y}$ containing $B_{2r}(y)$ satisfies
\[
B_{2r}(y)\subset\Om_y\Subset B_{4r}(y),
\quad
L_v\psi_y>0\geq L_v(v-u)
\quad\text{in}\ \Om_y\sm\ol{B_r(y)}.
\]
Therefore, the maximum principle gives
\[
v-u\geq\e\ \text{in}\ B_r(y)
\ \Gra\
v-u\geq\th\e\ \text{in}\ B_{2r}(y),
\]
where $\th\in(0,1)$ is a uniform constant.
Iteration along overlapping balls gives
\[
v-u\geq2\d\quad\text{in}\ B_{1/2}.
\]
Moreover, using the global H\"older estimate for harmonic functions,
we also obtain the boundary decay estimate $v-u\leq C\op{dist}(x,\p B_{3/2})^{1/2}$.
Thus, setting $w:=v-\d$,
the component $\Om$ of $\set{w>u}$ containing $B_{1/2}$ satisfies
\[
B_{1/2}\subset\Om\Subset B_{3/2},
\quad
w=u\ \text{on}\ \p\Om,
\quad
w-u\geq\d\ \text{in}\ B_{1/2}.
\]

Following this construction for the quadratic Hessian equation,
one can complete the proof
by applying the Chou--Wang Pogorelov estimate $(w-u)^\a|D^2u|\leq C$ on $\Om$.
However, for Hessian quotient equations, such an estimate fails in general.
As explained in \cite[page 3]{LT26b},
the quantities $\sum_{i=1}^nF^{ii}$ and $\sum_{i=1}^nF^{ii}\l_i^2$
do not generate a positive term growing with $\l_{\max}$
that can absorb the cutoff errors.
The counterexample in \cite[Proposition 9.2]{QY26b}
rules out a general fixed-level comparison-gap Pogorelov estimate
controlled only by the domain diameter and the $C^1$ norms.
Nevertheless,
we may apply the maximum principle to the auxiliary function
\[
\Phi(x,\xi)
:=(w-u)^2(x)\exp\frb{\frac\tau2|Du(x)|^2}u_{\xi\xi}(x)
\]
with $\tau>0$ sufficiently small
to obtain the doubling inequality
\begin{equation}\label{eqn.doubling1/2}
\sup_{B_{1/2}}\l_{\max}(D^2u)
\leq C\frb{n,k,\|u\|_{C^1(B_2)}}\sup_{B_{1/4}}\l_{\max}(D^2u).
\end{equation}
Since $\zeta$ vanishes outside $B_{1/4}$,
the concavity of $F^{1/(k-\ell)}$ gives
\[
L_u(w-u)
=L_u(v-u)
\geq-\frac{k-\ell}{4}\zeta
=0
\quad\text{in}\ \Om\sm\op{supp}\zeta.
\]
At a maximum point $(x_0,\xi_0)$ with $x_0\notin\ol{B_{1/4}}$,
we use this nonnegativity and $\G_k\subseteq\G_2$ to get
\[
0\geq
\frb{\frac\tau2-\frac{C}{\frb{(w-u)\l_{\max}}^2}}
\sum_{i=1}^nF^{ii}\l_i^2.
\]
Dividing by the positive factor $\sum_{i=1}^nF^{ii}\l_i^2$,
without requiring it to grow with $\l_{\max}$, we obtain
\[
(w-u)(x_0)\l_{\max}(D^2u(x_0))
\leq C\frb{n,k,\|u\|_{C^1(B_{3/2})}}.
\]
Combining this estimate with the bound from $B_{1/4}$
and using $w-u\geq\d$ in $B_{1/2}$
proves \eqref{eqn.doubling1/2}.

Finally, we conclude by applying the compactness argument
in \cite[Section 5]{W26} with \eqref{eqn.doubling1/2}.

\medskip

\section{Preliminaries}\label{sec.preliminaries}

In this section,
we fix notation and recall several standard facts
about the Hessian quotient equation
that will be used throughout the paper.

Let $\l=(\l_1,\l_2,\dots,\l_n)\in\R^n$.
For $i\in\Z_{[1,n]}$,
denote by $\l|i$ the vector obtained from $\l$ by deleting the $i$-th component.
By convention, $\s_i=0$ if $i<0$ or $i>n$, and $\s_0=1$.
For the Hessian quotient $F=\frac{\s_k}{\s_\ell}$
with $0\leq\ell<k\leq n$,
its normalized operator $F^{1/(k-\ell)}$ is positive, elliptic,
concave and homogeneous of degree one on $\G_k$; see \cite{CNS85}.
We write
\[
F_i:=\frac{\p F}{\p\l_i}
\quad\text{and}\quad
F_{ij}:=\frac{\p^2F}{\p\l_i\p\l_j}.
\]
For any $n\times n$ real symmetric matrix $M$ with its eigenvalue-vector $\l(M)\in\G_k$,
we write $F(M):=F(\l(M))$ for short, and then define
\[
F^{ij}(M)
:=\frac{\p F(M)}{\p M_{ij}},
\quad
F^{pq,rs}(M)
:=\frac{\p^2F(M)}{\p M_{pq}\p M_{rs}}.
\]
At a point where $D^2u$ is diagonal
with eigenvalues $\l_1,\l_2,\dots,\l_n$,
we have
\[
F^{ij}=F_i\d_{ij}\quad\text{and}\quad F^{ii,jj}=F_{ij}.
\]
For a smooth $k$-convex solution $u$ of $F(D^2u)=f$, we call
\[
L_u:=F^{ij}(D^2u)\p_{ij}
\]
the linearized operator of $F$ at $u$.
Note that $L_uu=F^{ij}(D^2u)u_{ij}=(k-\ell)F$.

\medskip

Now, we introduce some basic algebraic inequalities for $F$.

\begin{property}\label{prop.inequality}
Let $0\leq \ell<k\leq n$ and $\l\in\G_k$. Then
\begin{enumerate}[\qquad\quad(1)]
\item
$\ds F_i(\l)
=\frac{\s_{k-1}(\l|i)}{\s_\ell(\l)}
-\frac{\s_k(\l)\s_{\ell-1}(\l|i)}{\s_\ell(\l)^2}
>0$,

\item
$\ds\sum_{i=1}^nF_i
=(n-k+1)\frac{\s_{k-1}(\l)}{\s_\ell(\l)}
-(n-\ell+1)
\frac{\s_k(\l)\s_{\ell-1}(\l)}{\s_\ell(\l)^2}$.

\item
$\ds\sum_{i=1}^n\l_iF_i=(k-\ell)F$,
$\ds\sum_{i=1}^n\l_iF_{ij}=(k-\ell-1)F_j$ for any $j\in\Z_{[1,n]}$.
\end{enumerate}
\end{property}

\medskip

The following lower bounds for the linearized coefficients of $F$
will be used to construct a subsolution (Lemma \ref{lem.barrier})
and subsequently to prove the propagation lemma
(Lemma \ref{lem.propagation}).

\begin{lemma}\label{lem.Fj-lambdaj}
Let $2\leq k\leq n$, $0\leq\ell<k$, $k-\ell\in\set{1,2}$.
Let $\l\in\G_k$ satisfy $3/4\leq F^{1/(k-\ell)}(\l)\leq1$.
Then
\begin{equation}\label{eqn.Fj-lambdaj>}
F_j(\l)(1+\l_j^2)
\geq c_1(n,k)\sum_{i=1}^nF_i(\l)
\quad\text{for any}\ j\in\Z_{[1,n]}.
\end{equation}
Moreover, if $\l_j<0$, then
\begin{equation}\label{eqn.Fj>}
F_j(\l)\geq c_2(n,k)\sum_{i=1}^nF_i(\l),
\end{equation}
where $c_1=c_1(n,k)$ and $c_2=c_2(n,k)$
are positive constants.
\end{lemma}

\begin{proof}
On the level set $\{F=1\}$,
the estimates for $1\leq\ell\leq n$ follow from the proof of \cite[Lemma 2.2]{W26}.
When $\ell=0$, we have $k=2$ and $F=\s_2$,
so the estimates follow from \cite[Lemma 2.1]{LW26b}
and the identity $\sum_{i=1}^nF_i=(n-1)\s_1$.
Let $c_1^*=c_1^*(n,k)$ and $c_2^*=c_2^*(n,k)$
denote the corresponding positive constants on $\{F=1\}$.
Let
\[
A:=F(\l)^{1/(k-\ell)},
\quad
\mu:=\frac{\l}{A}.
\]
Since $F$ is homogeneous of degree $k-\ell$, we have
\[
\mu\in\G_k,
\quad
F(\mu)=1,
\quad
F_i(\mu)=A^{1-(k-\ell)}F_i(\l).
\]
Applying \cite[Lemma 2.2]{W26} to $\mu$, we obtain
\[
F_j(\mu)(1+\mu_j^2)
\geq c_1^*\sum_{i=1}^nF_i(\mu)
\quad\text{and}\quad
F_j(\mu)\geq c_2^*\sum_{i=1}^nF_i(\mu)
\quad\text{if}\ \l_j<0.
\]
Rescaling back to $\l$ and using $3/4\leq A\leq1$, we get
\begin{align*}
F_j(\l)(1+\l_j^2)
&\geq F_j(\l)(A^2+\l_j^2)
=A^{k-\ell+1}F_j(\mu)(1+\mu_j^2)
\\
&\geq c_1^*A^2\sum_{i=1}^nF_i(\l)
\geq\frac9{16}c_1^*\sum_{i=1}^nF_i(\l).
\\
F_j(\l)
&\geq c_2^*\sum_{i=1}^nF_i(\l)
\quad\text{if}\ \l_j<0.
\end{align*}
This completes the proof with $c_1:=\frac9{16}c_1^*$ and $c_2:=c_2^*$.
\end{proof}

\medskip

Finally, we collect two classical results for the Hessian quotient equation:
the solvability of the Dirichlet problem,
and the interior gradient estimate.

\begin{lemma}[Classical solvability,
\texorpdfstring{\cite[Theorem 1.1]{Tru95}}{}]
\label{lem.solvability}
Let $0\leq \ell<k\leq n$,
$r>0$, $f\in C^{1,1}(\ol{B_r})$ with $f>0$ in $\ol{B_r}$
and $g\in C^{3,1}(\p B_r)$.
Then for any $\a\in(0,1)$,
there exists a unique $k$-convex solution $u\in C^{3,\a}(\ol{B_r})$ to the Dirichlet problem
\begin{equation}\label{eqn.Dirichlet}
F(D^2u)=f\ \textrm{in}\ B_r,
\quad
u=g\ \textrm{on}\ \p B_r.
\end{equation}
\end{lemma}

\begin{lemma}[Gradient estimate,
\texorpdfstring{\cite[Theorem 1.1]{C15}}{}]
\label{lem.gradient-estimate}
Let $0\leq \ell<k\leq n$, $r>0$, and $f\in C^1(B_r)$ with $f>0$ in $B_r$.
Suppose $u\in C^3(B_r)$ is a $k$-convex solution of $F(D^2u)=f$ in $B_r$.
Then
\[
\|Du\|_{L^\infty(B_{r/2})}
\leq C\frb{n,k,\ell,\|Df\|_{L^\infty(B_r)}}
\frb{\frac{\op{osc}_{B_r}u}{r}
+\frb{\op{osc}_{B_r}u}^{\frac{k-l+1}{2(k-l)}}
+\frb{\op{osc}_{B_r}u}^{\frac{k-l}{2(k-l)+1}}}.
\]
\end{lemma}

\section{Proof of the theorem}\label{sec.doubling}

Throughout this section, let $u$ be a smooth $k$-convex solution
of \eqref{eqn.sigma-quotient} in $B_2$, and set
\[
K_0:=1+\norm{u}_{C^1(B_2)},
\quad b:=\l_{\max}(D^2u).
\]
Choose a fixed cutoff $\zeta\in C_0^\infty(B_2)$
such that $0\leq\zeta\leq1$ in $B_2$, $\zeta=1$ in $B_{1/32}$,
$\op{supp}\zeta\subset B_{1/4}$,
and $\|D\zeta\|_{L^\infty(B_1)}\leq C(n)$.
By Lemma \ref{lem.solvability},
there exists a smooth $k$-convex solution $v$ of
\begin{equation}\label{eqn.v}
\begin{cases}
F(D^2v)=\frb{1-\zeta/4}^{k-\ell}&\text{in }B_{3/2},\\
v=u&\text{on }\p B_{3/2}.
\end{cases}
\end{equation}
The comparison principle gives $u\leq v$ in $B_{3/2}$.
Since $\l(D^2v)\in\G_k$, we have $\t v>0$ in $B_{3/2}$,
and hence the maximum principle gives
\[
\sup_{B_{3/2}}v\leq\sup_{\p B_{3/2}}v=\sup_{\p B_{3/2}}u\leq K_0.
\]
On the other hand, $v\geq u\geq-K_0$ in $B_{3/2}$.
Therefore $\norm{v}_{L^\infty(B_{3/2})}\leq K_0$.
By the gradient estimate (Lemma \ref{lem.gradient-estimate}),
we choose $K=K(n,k,K_0)\geq1$ such that
\begin{equation}\label{eqn.gradient-v-u}
\norm{Dv}_{L^\infty(B_1)}
+\norm{u}_{C^1(B_2)}
+\norm{\zeta}_{C^1(B_{3/2})}\leq K.
\end{equation}
In particular, $|D(v-u)|\leq K$ in $B_1$.

\medskip

We next quantify the separation between $v$ and $u$.

\subsection{A seed separation}

We adapt the seed argument of \cite[Lemma 3.1]{LW26b}.

\begin{lemma}\label{lem.seed-F}
There exist constants $r=r(n,k,K)\in(0,1/32)$, $\ve_0=\ve_0(n,k,K)>0$,
and a point $x_*\in\ol{B_{1/32}}$ such that
\[
v-u\geq\ve_0
\quad\text{in}\ \ol{B_r(x_*)}.
\]
\end{lemma}

\begin{proof}
Fix $\rho:=1/32$.
Since $\zeta=1$ in $B_\rho$, we have $F(D^2v)=\frb{3/4}^{k-\ell}$ in $B_\rho$.
By the concavity of $F^{1/(k-\ell)}$, we obtain
\begin{align}
\sum_{i,j=1}^nF^{ij}(D^2v)(v-u)_{ij}
&\leq(k-\ell)F(D^2v)^{1-\frac{1}{k-\ell}}
\frb{F(D^2v)^{\frac1{k-\ell}}-F(D^2u)^{\frac1{k-\ell}}}
\nonumber\\
&=-\frac{k-\ell}{4}\frb{\frac34}^{k-\ell-1}
=:-\varsigma<0.
\label{eqn.polarized}
\end{align}
Write $\s_m^{ij}(M):=\frac{\p\s_m(M)}{\p M_{ij}}$ for any $m\in\Z_{[0,k]}$ and $\s_0^{ij}:=0$.
Set
\[
\wt F^{ij}
:=\s_\ell(D^2v)F^{ij}(D^2v)
=\s_k^{ij}(D^2v)-F\s_\ell^{ij}(D^2v)
\quad\text{in}\ B_\rho.
\]
Then $(\wt F^{ij})>0$.
Differentiating and using the symmetry of the third derivatives,
we obtain
\begin{equation}\label{eqn.DwtF=0}
\sum_{i=1}^n\p_i\s_m^{ij}(D^2v)=0,
\quad
\sum_{i=1}^n\p_i\wt F^{ij}=0\quad\text{for all}\ j\in\Z_{[1,n]}.
\end{equation}
Moreover,
by Property \ref{prop.inequality}, it follows that
\begin{equation}\label{eqn.Fii-upper}
0<\sum_{i=1}^n\wt F^{ii}\leq(n-k+1)\s_{k-1}(D^2v).
\end{equation}
The Newton--Maclaurin inequalities and $F(D^2v)=\frb{3/4}^{k-\ell}$
imply
\begin{equation}\label{eqn.sigma-k-1>}
\s_\ell(D^2v)\geq c(n,k)>0\quad\text{in}\ B_\rho.
\end{equation}

Next, choose $\eta\in C_c^\infty(B_1)$ satisfying
$\eta\geq0$ in $B_1$, $\int_{B_1}\eta=1$ and $\norm{D^2\eta}_{L^\infty(B_1)}\leq C(n)$,
and set $\eta_\rho(x):=\rho^2\eta(x/\rho)$.
Then $\eta_\rho\in C_c^\infty(B_\rho)$, $\int_{B_\rho}\eta_\rho=\rho^{n+2}$
and $\norm{D^2\eta_\rho}_{L^\infty(B_\rho)}\leq C(n)$.
Multiplying \eqref{eqn.polarized} by $\eta_\rho\s_\ell(D^2v)$,
integrating by parts twice,
and using \eqref{eqn.Fii-upper}, \eqref{eqn.DwtF=0} and \eqref{eqn.sigma-k-1>}, we obtain
\begin{align*}
c\varsigma\rho^{n+2}
&\leq\varsigma\int_{B_\rho}\eta_\rho\s_\ell(D^2v)
\leq-\int_{B_\rho}\eta_\rho\sum_{i,j=1}^n\wt F^{ij}(v-u)_{ij}
\\
&=\int_{B_\rho}\sum_{i,j=1}^n\wt F^{ij}(\eta_\rho)_j(v-u)_i
=-\int_{B_\rho}(v-u)\sum_{i,j=1}^n\wt F^{ij}(\eta_\rho)_{ij}
\\
&\leq C(n)\sup_{B_\rho}(v-u)
\int_{B_\rho}\sum_{i=1}^n\wt F^{ii}
\\
&\leq C(n,k)\sup_{B_\rho}(v-u)
\int_{B_\rho}\s_{k-1}(D^2v).
\end{align*}
We now estimate $\int_{B_\rho}\s_{k-1}(D^2v)$.
Set $\rho_j:=\rho+\frac{j\rho}{k-1}$ for $j\in\Z_{[0,k-1]}$.
Then $\rho_0=\rho$, $\rho_{k-1}=2\rho$, and $B_{2\rho}\subset B_1$.
For each $j\in\Z_{[0,k-2]}$, let $m:=k-1-j\geq1$
and choose $\xi_j\in C_0^\infty(B_{\rho_{j+1}})$
such that $0\leq\xi_j\leq1$ in $B_{\rho_{j+1}}$, $\xi_j=1$ in $B_{\rho_j}$,
and $\abs{D\xi_j}\leq\frac{C(n,k)}{\rho}$ in $B_{\rho_{j+1}}$.
Using the Newton-tensor identities and $|Dv|\leq K$ in $B_1$,
we obtain
\begin{align*}
\int_{B_{\rho_j}}\s_m(D^2v)
&\leq\int_{B_{\rho_{j+1}}}\xi_j\s_m(D^2v)
=\frac{1}{m}\int_{B_{\rho_{j+1}}}\xi_j\sum_{p,q=1}^n\s_m^{pq}(D^2v)v_{pq}
\\
&=-\frac{1}{m}\int_{B_{\rho_{j+1}}}\sum_{p,q=1}^n(\xi_j)_p\s_m^{pq}(D^2v)v_q
\leq\frac{C(n,k)K}{\rho}\int_{B_{\rho_{j+1}}}\sum_{p=1}^n\s_m^{pp}(D^2v)
\\
&=\frac{C(n,k)K(n-m+1)}{\rho}\int_{B_{\rho_{j+1}}}\s_{m-1}(D^2v).
\end{align*}
Iterating gives
\begin{align*}
\int_{B_\rho}\s_{k-1}(D^2v)
\leq\frac{C}{\rho}\int_{B_{\rho_1}}\s_{k-2}(D^2v)
\leq\cdots
\leq\frb{\frac{C}{\rho}}^{k-1}\int_{B_{2\rho}}\s_0(D^2v)
\leq C(n,k,K)\rho^{n-k+1}.
\end{align*}
Combining this with the preceding estimate gives
\[
\sup_{B_\rho}(v-u)
\geq c(n,k,K)\rho^{k+1}=:c_0>0.
\]
Choose $x_*\in\ol{B_\rho}$ such that $v(x_*)-u(x_*)\geq c_0$.
Take $r=r(n,k,K)\in(0,1/32)$ sufficiently small that
$Kr\leq c_0/2$, and set $\ve_0:=c_0/2$.
Then $B_r(x_*)\subset B_{1/16}$, and the gradient bound \eqref{eqn.gradient-v-u} gives
\begin{align*}
v(x)-u(x)
&\geq v(x_*)-u(x_*)
-\norm{D(v-u)}_{L^\infty(B_1)}|x-x_*|
\\
&\geq c_0-Kr\geq\ve_0
\quad\text{for any}\ x\in B_r(x_*).
\end{align*}
This completes the proof.
\end{proof}

\begin{remark}\label{rem.seed-compactness}
The seed lemma also follows from a compactness argument,
without using the divergence-free structure of the Newton tensors.
By the gradient estimate,
it suffices to show $\sup_{B_\rho}(v-u)\geq c(n,k,K)>0$.
Suppose to the contrary that
there exist a sequence of $k$-convex solutions $u_j$ of \eqref{eqn.sigma-quotient}
and the corresponding solutions $v_j$ of \eqref{eqn.v} such that
\[
\norm{u_j}_{C^1(B_2)}\leq K,
\quad
0\leq v_j-u_j\leq\frac{1}{j}
\quad\text{in}\ B_\rho.
\]
By the Arzel\`a--Ascoli theorem,
after passing to a subsequence,
$u_j\to u_\infty$ locally uniformly in $B_2$.
Moreover, $v_j\to u_\infty$ locally uniformly in $B_\rho$.
The closedness of viscosity solutions gives
\[
F(D^2u_\infty)=1,
\quad
F(D^2u_\infty)=\frb{3/4}^{k-\ell}
\quad\text{in}\ B_\rho\ \text{in the viscosity sense}.
\]
By the Alexandrov-type differentiability \cite[Lemma 5.1]{W26} when $\ell\geq1$
(or \cite[Proposition 4.1]{SY25} when $\ell=0$),
together with the two viscosity equations,
there exist $y\in B_\rho$ and a quadratic polynomial $P$
such that $u_\infty(y+z)=P(y+z)+o(|z|^2)$,
$\l(D^2P)\in\G_k$, and $1=F(D^2P)=\frb{3/4}^{k-\ell}$ at $y$,
a contradiction.
\end{remark}

\medskip

We next construct a barrier for $L_v$ to propagate this lower bound.

\subsection{A barrier for the linearized operator \texorpdfstring{$L_v$}{Lv}}\label{subsec.barrier}

Keep $r=r(n,k,K)\in(0,1/32)$ from Lemma \ref{lem.seed-F}.
For $y\in B_{1/2}$, we have $\ol{B_{4r}(y)}\subset B_{3/2}$.
Consider the auxiliary function
\[
\vp_y(x):=(x-y)\cdot Dv(x)-v(x)+v(y)
+\frac{\a}{2}|x-y|^2-\b|Dv(x)|^2,
\]
where $\a=\a(n,k,r,K)\geq1$ and $\b=\b(n,r,K)>0$ will be chosen below.
The gradient estimate \eqref{eqn.gradient-v-u} gives
\[
\abs{(x-y)\cdot Dv(x)-v(x)+v(y)}\leq2K|x-y|
\quad\text{for any}\ x\in\ol{B_{4r}(y)}.
\]
Then
\[
\vp_y\geq-8Kr-\b K^2\quad\text{in}\ B_{4r}(y),\quad
a_y:=\sup_{\ol{B_{2r}(y)}}\vp_y
\leq4Kr+2\a r^2.
\]
Choosing $\a$ such that $6r^2\a-K^2\b-12Kr-2>0$,
we have
\[
\inf_{\p B_{4r}(y)}\vp_y
\geq-8Kr+\frac{\a}{2}(4r)^2-\b K^2
>4Kr+2\a r^2+2
\geq a_y+2=:c_y.
\]
Hence
\[
\vp_y<c_y
\quad\text{in}\ \ol{B_{2r}(y)},
\quad
\vp_y>c_y
\quad\text{on}\ \p B_{4r}(y).
\]
Thus the connected component $\Om_y$ of
$\{x\in B_{4r}(y):\vp_y(x)<c_y\}$
containing $B_{2r}(y)$ satisfies
\begin{align*}
\ol{B_{2r}(y)}\subset\Om_y\Subset B_{4r}(y).
\end{align*}
Next, define the associated exponential auxiliary function
\begin{align*}
\psi_y:=e^{(c_y-\vp_y)/\g}-1
\quad\text{in}\ \Om_y,
\end{align*}
where $\g=\g(n,k,r,K)>0$ will be chosen later.
Then $\psi_y>0$ in $\Om_y$ and $\psi_y=0$ on $\p\Om_y$.
Moreover,
\begin{gather}
\psi_y\geq e^{1/\g}-1=:c_0>0
\quad\text{in}\ \ol{B_{2r}(y)}.
\label{eqn.psi-lower}
\\
\psi_y\leq
\exp\frb{\frac{12Kr+2\a r^2+\b K^2+2}{\g}}-1
=:C_0
\quad\text{in}\ \Om_y.
\label{eqn.psi-upper}
\end{gather}

\begin{lemma}\label{lem.barrier}
For any $y\in B_{1/2}$,
the parameters $\a,\b,\g$ can be chosen depending only on $n,k,K$
such that
\begin{equation}\label{eqn.barrier-positive}
L_v\psi_y>0\quad\text{in}\ \Om_y\sm B_r(y).
\end{equation}
\end{lemma}

\begin{proof}

After a rotation of coordinates, we may assume that
\[
D^2v(x)
=\op{diag}(\l_1,\l_2,\ldots,\l_n),
\quad
\l_1\geq\l_2\geq\cdots\geq\l_n.
\]
All quantities below are evaluated at $x$.
By the homogeneity and differentiation of equation $F(D^2v)=\frb{1-\zeta/4}^{k-\ell}$,
we have
\begin{gather*}
L_vv=(k-\ell)\frb{1-\frac\zeta4}^{k-\ell},
\\
\sum_{i,j=1}^nF^{ij}(D^2v)v_{ijl}
=-(k-\ell)\frb{1-\frac\zeta4}^{k-\ell-1}\frac{\zeta_l}{4}
\quad\text{for all}\ l\in\Z_{[1,n]}.
\end{gather*}
Consequently,
\begin{align*}
(\vp_y)_i
&=\frb{x_i-y_i-2\b v_i}\l_i+\a(x_i-y_i),
\\
(\vp_y)_{ij}
&=v_{ij}+\sum_{l=1}^n(x_l-y_l-2\b v_l)v_{ijl}+\a\d_{ij}-2\b\sum_{l=1}^n v_{li}v_{lj},
\\
L_v\psi_y
&=\frac{\psi_y+1}{\g}
\bigg[\frac{1}{\g}\sum_{i=1}^nF_i(\vp_y)_i^2
+2\b\sum_{i=1}^nF_i\l_i^2
-\a\sum_{i=1}^nF_i
\\
&\qquad
-(k-\ell)\frb{1-\frac{\zeta}{4}}^{k-\ell-1}
\frb{1-\frac{\zeta}{4}
-\frac{1}{4}(x-y-2\b Dv)\cdot D\zeta}
\bigg].
\end{align*}
Set $\mathds 1:=(1,1,\ldots,1)\in\R^n$.
Since $F^{1/(k-\ell)}$ is concave and homogeneous in $\G_k$,
we have
\begin{align*}
\frac{1}{c_*}:=F^{1/(k-\ell)}(\mathds1)
&\leq F^{1/(k-\ell)}(\l)+DF^{1/(k-\ell)}(\l)\cdot(\mathds1-\l)
\\
&=\sum_{i=1}^nF^{1/(k-\ell)}_i(\l)
=\frac{\sum_{i=1}^nF_i}{(k-\ell)\frb{1-\zeta/4}^{k-\ell-1}},
\end{align*}
and hence
\[
\sum_{i=1}^nF_i\geq\frac{(k-\ell)}{c_*}\frb{1-\frac{\zeta}{4}}^{k-\ell-1}.
\]
Moreover, the bounds $0\leq\zeta\leq1$, $|Dv|+|D\zeta|\leq K$ give
\[
1-\frac{\zeta}{4}
-\frac14(x-y-2\b Dv)\cdot D\zeta
\leq1+\frb{r+\frac{\b K}{2}}K=:c_3.
\]
Consequently, setting $\a_1:=\a+c_*c_3$,
we obtain
\begin{equation}\label{eqn.Lvpsi-lower}
L_v\psi_y
\geq\frac{\psi_y+1}{\g}
\frb{\frac{1}{\g}\sum_{i=1}^nF_i(\vp_y)_i^2+2\b\sum_{i=1}^nF_i\l_i^2
-\a_1\sum_{i=1}^nF_i}.
\end{equation}
Therefore, to prove $L_v\psi_y>0$, it suffices to prove
\begin{equation}\label{eqn.Q>Fi}
Q:=\frac{1}{\g}\sum_{i=1}^nF_i(\vp_y)_i^2+2\b\sum_{i=1}^nF_i\l_i^2
>\a_1\sum_{i=1}^nF_i.
\end{equation}
Since $|x-y|\geq r$, there exists $j\in\Z_{[1,n]}$ such that
\[
|x_j-y_j|\geq|x-y|/\sqrt n\geq r/\sqrt n.
\]
Set $z_j:=x_j-y_j-2\b v_j$.
Then $(\vp_y)_j=z_j\l_j+\a(x_j-y_j)$.
Choosing $\b:=\frac{r}{16K\sqrt n}$, we have
\[
2\b|v_j|\leq2\b K
=\frac{r}{8\sqrt n}
\leq\frac18|x_j-y_j|.
\]
Hence
\begin{gather}
(x_j-y_j)z_j
=(x_j-y_j)^2-2\b\frb{x_j-y_j}v_j
\geq\frac78|x_j-y_j|^2>0,
\label{eqn.zj-bound}
\\
\frac78|x_j-y_j|\leq|z_j|
\leq\frac98|x_j-y_j|.
\nonumber
\end{gather}
Next, we distinguish two cases.

\textit{Case 1. $\l_j\geq0$}.
By \eqref{eqn.zj-bound},
the terms $z_j\l_j$ and $\a(x_j-y_j)$ have the same sign.
By $\a\geq1$ and Lemma \ref{lem.Fj-lambdaj},
there exists $c_1=c_1(n,k)>0$ such that
\[
Q
\geq\frac{F_j}{\g}(\vp_y)_j^2
\geq\frac{F_j}{\g}\frb{\frac78\l_j+\a}^2|x_j-y_j|^2
\geq\frac{49r^2}{64n\g}F_j(1+\l_j^2)
\geq\frac{49r^2c_1}{64n\g}\sum_{i=1}^nF_i.
\]
Hence \eqref{eqn.Q>Fi} holds whenever $\g<\frac{49r^2c_1}{64n\a_1}$.

\textit{Case 2. $\l_j<0$}.
Take $\g$ with $\g\leq\frac{r^2}{2\b n}=\frac{8Kr}{\sqrt n}$.
Then $2\b\g\leq(x_j-y_j)^2$.
By Lemma \ref{lem.Fj-lambdaj},
there exists $c_2=c_2(n,k)>0$ such that
\begin{align*}
Q
\geq\frac{1}{\g}F_j(\vp_y)_j^2+2\b F_j\l_j^2
\geq\frac{2\b\a^2(x_j-y_j)^2}{z_j^2+2\b\g}F_j
\geq\frac{1}{2}\b\a^2F_j
\geq\frac{1}{2}\b\a^2 c_2\sum_{i=1}^nF_i.
\end{align*}
Choose $\a\geq\max\set{1,c_*c_3,\frac{8}{\b c_2}}$.
Then $\frac{1}{2}\b\a^2 c_2>2\a\geq\a+c_*c_3=\a_1$, and hence \eqref{eqn.Q>Fi} holds.
Since $\sum_{i=1}^nF_i>0$,
\eqref{eqn.Lvpsi-lower} and \eqref{eqn.Q>Fi} yield $L_v\psi_y(x)>0$.
\end{proof}

\medskip

Since $L_v(v-u)\leq0$,
Lemma \ref{lem.barrier} allows us to use $\psi_y$ as a barrier for $v-u$.
The comparison principle then yields the quantitative
propagation of lower bounds from $B_r(y)$ to $B_{2r}(y)$.

\subsection{Propagation of the separation}

\begin{lemma}\label{lem.propagation}
Let $y\in B_{1/2}$ and $\e>0$.
If $v-u\geq\e$ in $B_r(y)$, then
\[
v-u\geq\th\e\quad\text{in }B_{2r}(y),
\]
where $\th=\th(n,k,K)\in(0,1)$ is independent of $\e$.
\end{lemma}

\begin{proof}
Let $U_y$ be the connected component of $\Om_y\sm\ol{B_r(y)}$
containing $B_{2r}(y)\sm\ol{B_r(y)}$.
Then
\[
\p U_y\subset\p\Om_y\cup\p B_r(y).
\]
Recalling the lower bound \eqref{eqn.psi-lower} and the upper bound \eqref{eqn.psi-upper} for $\psi_y$,
we have $v-u\geq\e\geq\frac{\e}{C_0}\psi_y$ on $\p U_y\cap\p B_r(y)$.
Moreover,
$v-u\geq0=\psi_y$ on $\p U_y\cap\p\Om_y$.
Consequently,
\[
(v-u)\geq \frac{\e}{C_0}\psi_y\quad\text{on}\ \p U_y.
\]
By the concavity of $F^{1/(k-\ell)}$ and the equations for $u$ and $v$,
\begin{align*}
L_v(v-u)
&=F^{ij}(D^2v)(v-u)_{ij}
\\
&\leq(k-\ell)\frb{F(D^2v)}^{1-\frac1{k-\ell}}
\frb{F^{\frac1{k-\ell}}(D^2v)
-F^{\frac1{k-\ell}}(D^2u)}
\\
&=-(k-\ell)\frb{1-\frac{\zeta}{4}}^{k-\ell-1}
\frac{\zeta}{4}
\leq0
\quad\text{in}\ B_{3/2}.
\end{align*}
Together with Lemma \ref{lem.barrier}, this gives
\[
L_v(v-u)\leq0<\frac{\e}{C_0}L_v\psi_y
\quad\text{in}\ U_y.
\]
The maximum principle therefore yields $v-u\geq\frac{\e}{C_0}\psi_y$ in $\ol{U_y}$.
Since $\psi_y\geq c_0$ in $\ol{B_{2r}(y)}$, we obtain
\[
v-u\geq\frac{\e}{C_0}\psi_y\geq\frac{c_0}{C_0}\e
\quad\text{in}\ B_{2r}(y)\sm\ol{B_r(y)}.
\]
Set $\th:=\min\set{\frac{1}{2},\frac{c_0}{C_0}}$.
Then $v-u\geq\e\geq\th\e$ in $\ol{B_r(y)}$.
Thus $v-u\geq\th\e$ in $\ol{B_{2r}(y)}$.
\end{proof}

\subsection{Comparison domain and comparison function}\label{subsec.comparison}

Using the seed lemma (Lemma \ref{lem.seed-F})
and then iterating the propagation lemma (Lemma \ref{lem.propagation}),
we obtain a uniform positive lower bound for $v-u$ in $B_{1/2}$.

\begin{lemma}[Uniform interior separation]\label{lem.separation}
There exists $\ve_1=\ve_1(n,k,K)>0$ such that
\begin{equation}\label{eqn.uniform-separation}
v-u\geq\ve_1\quad\text{in}\ B_{1/2}.
\end{equation}
\end{lemma}

\begin{proof}
Iterating the propagation lemma
along a finite chain of overlapping balls gives the conclusion,
as in \cite[Lemma  3.4]{LW26b}.
\end{proof}

\medskip

The decay estimate shows that $\{v-u>\d\}$ stays a uniform distance away from $\p B_{3/2}$.

\begin{lemma}\label{lem.boundary-decay}
There exists $C=C(n,K)$ such that
\begin{equation}\label{eqn.boundary-decay}
0\leq (v-u)(x)
\leq C d_x^{1/2}
\quad\text{for any}\ x\in B_{3/2},
\end{equation}
where $d_x$ denotes the distance from $x$ to $\p B_{3/2}$.
\end{lemma}

\begin{proof}
The harmonic comparison argument in \cite[Lemma 3.5]{LW26b} applies here,
since $v$ is subharmonic and $v=u$ on $\p B_{3/2}$.
\end{proof}

\medskip

Combining these two lemmas,
we obtain a comparison domain $\Om$ and a comparison function $w$,
so that $w-u$ serves as a cutoff function in the doubling argument.

\begin{lemma}\label{lem.comparison-domain}
Let $\ve_1=\ve_1(n,k,K)$ be as in Lemma \ref{lem.separation},
$\d:=\ve_1/2$,
and let $w:=v-\d$ in $B_{3/2}$.
Suppose that $\Om$ is a connected component of $\set{w>u}$ containing $B_{1/2}$.
Then there exists $d_0=d_0(n,K)\in(0,1/4)$ such that
\begin{gather*}
\Om\Subset B_{3/2-d_0},
\quad
w-u>0\ \text{in}\ \Om,
\quad
w=u\ \text{on}\ \p\Om,
\\
w-u\geq\d\ \text{in}\ B_{1/2},
\quad
\|Dw\|_{L^\infty(\Om)}\leq C(n,k,K).
\end{gather*}
\end{lemma}

\begin{proof}
Choose $d_0=d_0(n,K)\in(0,1/4)$ sufficiently small
such that
\[
C(n,K)d_0^{1/2}\leq\ve_1/4,
\]
where $C(n,K)$ comes from Lemma \ref{lem.boundary-decay}.
By Lemma \ref{lem.separation},
$v-u\geq\ve_1>\d$ in $B_{1/2}$,
hence $B_{1/2}\subset\set{v-u>\d}=\set{w>u}$.
For any $x\in B_{3/2}$,
if $d_x\leq d_0$, then $|x|=\frac32-d_x\geq\frac32-d_0>\frac54$.
Hence it follows from Lemma \ref{lem.boundary-decay} that
\begin{equation}
(v-u)(x)\leq\frac{\ve_1}{4},
\quad
\text{whenever}\ d_x\leq d_0.
\end{equation}
Thus, the set $\set{v-u>\d}=\set{w>u}$ cannot intersect $\set{d_x\leq d_0}$.
This shows that $\Om\Subset B_{3/2-d_0}$,
and hence $B_{d_0/2}(x)\subset B_{3/2}$ for any $x\in\Om$.
Therefore,
\[
w>u\quad\text{in}\ \Om,\
\quad
w=u\quad\text{on}\ \p\Om,\
\quad
w-u\geq\d\quad\text{in}\ B_{1/2}.
\]
Applying the gradient estimate (Lemma \ref{lem.gradient-estimate}) in $B_{d_0/2}(x)$, we get
\[
|Dv(x)|
\leq C(n,k,K_0,d_0)\norm{v}_{L^\infty(B_{d_0/2}(x))}
\leq C(n,k,K).
\]
Thus
\[
\norm{Dw}_{L^\infty(\Om)}=\norm{Dv}_{L^\infty(\Om)}\leq C(n,k,K).
\]
This completes the proof.
\end{proof}

\subsection{The doubling inequality}\label{subsec.doubling}

\begin{lemma}[Doubling inequality]\label{lem.doubling}
Let $b:=\l_{\max}(D^2u)$.
Then
\[
\sup_{B_{1/2}}b
\leq C(n,k,K_0)\sup_{B_{1/4}}b.
\]

Furthermore, for any $y\in B_{1/2}$ and $r\in(0,1/8)$, we have
\begin{equation}\label{eqn.doubling}
\sup_{B_{2r}(y)}b
\leq C(n,k,K_0,r)\sup_{B_r(y)}b.
\end{equation}
\end{lemma}

\begin{proof}
Set $\vp:=w-u$ in $\Om$.
Consider the auxiliary function
(similar to those in \cite[(7.4)]{QY26b} and \cite[page 10]{LW26b})
\[
\Phi(x,\xi):=\vp^2(x)
\exp\frb{\frac\tau2|Du(x)|^2}u_{\xi\xi}(x)
\quad\text{for any}\ (x,\xi)\in\ol\Om\times\mathbb S^{n-1},
\]
where $\tau\in(0,1]$ will be chosen later.
Since $\Phi=0$ on $\p\Om\times\mathbb S^{n-1}$,
$\Phi$ attains a positive maximum
at some $(x_0,\xi_0)\in\Om\times\mathbb S^{n-1}$.
Next, we distinguish two cases.

\textit{Case 1. $x_0\in\ol{B_{1/4}}$.}
Since $w-u\leq2K_0$ and $|Du|\leq K$,
\begin{equation}\label{eqn.inner-maximum}
\Phi(x_0,\xi_0)\leq C(n,k,K)\sup_{B_{1/4}}b.
\end{equation}

\textit{Case 2. $x_0\notin\ol{B_{1/4}}$.}
Since $\op{supp}\zeta\Subset B_{1/4}$, we have $\zeta(x_0)=0$.
The concavity of $F^{1/(k-\ell)}$ gives
\begin{align}
L_u\vp
=L_u(w-u)
&=L_u(v-u)
=F^{ij}(D^2u)(v-u)_{ij}
\nonumber\\
&\geq(k-\ell)\frb{F(D^2u)}^{1-\frac1{k-\ell}}
\frb{F^{\frac1{k-\ell}}(D^2v)-F^{\frac1{k-\ell}}(D^2u)}
\nonumber\\
&=-\frac{k-\ell}{4}\zeta(x_0)
=0.
\label{eqn.Luvp>=0}
\end{align}
After a rotation of coordinates,
we may assume that $\xi_0=e_1$ and
\[
D^2u(x_0)
=\op{diag}(\l_1,\l_2,\ldots,\l_n),
\quad
\l_1=\l_2=\cdots=\l_m>\l_{m+1}\geq\cdots\geq\l_n.
\]
All subsequent quantities are evaluated at $x_0$.
Since $u_{11}(x_0)>0$, the smooth function
\[
\Psi(x):=\log\Phi(x,e_1)
=2\log\vp+\frac\tau2|Du|^2+\log u_{11}
\]
attains a local maximum at $x_0$.
Thus
\begin{align}
0
&=\Psi_i
=2\frac{\vp_i}{\vp}+\tau u_i \l_i+\frac{u_{11i}}{\l_1}
\quad
\text{for all}\ i\in\Z_{[1,n]},
\label{eqn.Pogorelov-first}
\\
0
\geq F^{ij}\Psi_{ij}
&=\frac{2F^{ii}\vp_{ii}}{\vp}
-\frac{2F^{ii}\vp_i^2}{\vp^2}
+\tau u_kF^{ii}u_{kii}
+\tau F^{ii}\l_i^2
+\frac{F^{ii}u_{11ii}}{\l_1}
-\frac{F^{ii}u_{11i}^2}{\l_1^2}
\nonumber\\
&\geq
-\frac{2F^{ii}\vp_i^2}{\vp^2}
+\tau F^{ii}\l_i^2
+\frac{F^{ii}u_{11ii}}{\l_1}
-\frac{F^{ii}u_{11i}^2}{\l_1^2},
\label{eqn.Pogorelov-second}
\end{align}
where the last inequality follows from \eqref{eqn.Luvp>=0}
and $F^{ii}u_{iil}=\p_{x_l}F=0$ for all $l\in\Z_{[1,n]}$.
Differentiating the equation $F(D^2u)=1$ twice in the $x_1$-direction
and using the concavity of $F^{1/(k-\ell)}$,
we obtain
\begin{align}
F^{ii}u_{ii11}
=-F^{ij,rs}u_{ij1}u_{rs1}
&=-F^{ii,jj}u_{ii1}u_{jj1}
+2\sum_{i<j}\frac{F^{jj}-F^{ii}}{\l_i-\l_j}u_{ij1}^2
\geq2\sum_{i>m}\frac{F^{ii}-F^{11}}{\l_1-\l_i}u_{11i}^2.
\label{eqn.Fii-uii11}
\end{align}
For any $i>m$, since $\l_1>0$ and $\l_1-\l_i>0$, we have
\begin{equation}\label{eqn.lambda-ratio}
\frac{2\l_1}{\l_1-\l_i}
\geq2+\frac{2\l_i}{\l_1}
\geq\frac{7}{4}-\frac{4\l_i^2}{\l_1^2}.
\end{equation}
Thus, using $F^{ii}\geq F^{11}>0$ for all $i>m$
and $F^{11}=\frac{F^{ii}\l_i^2}{\l_1^2}$ for all $i\leq m$,
we have
\begin{align*}
2\sum_{i>m}\frac{F^{ii}-F^{11}}{\l_1(\l_1-\l_i)}u_{11i}^2
-\sum_{i=1}^nF^{ii}\frac{u_{11i}^2}{\l_1^2}
\geq\sum_{i=1}^n\frb{\frac{3}{4}F^{ii}-\frac{7}{4}F^{11}-\frac{4}{\l_1^2}F^{ii}\l_i^2}
\frac{u_{11i}^2}{\l_1^2}.
\end{align*}
Substituting this inequality and \eqref{eqn.Fii-uii11} into \eqref{eqn.Pogorelov-second},
we obtain
\begin{align}
0
\geq-\frac{2F^{ii}\vp_i^2}{\vp^2}
+\tau F^{ii}\l_i^2
+\frac{3}{4}F^{ii}\frac{u_{11i}^2}{\l_1^2}
-\frac{7}{4}F^{11}\sum_{i=1}^n\frac{u_{11i}^2}{\l_1^2}
-\frac{4}{\l_1^2}F^{ii}\l_i^2\frac{u_{11i}^2}{\l_1^2}.
\label{eqn.Pogorelov-third}
\end{align}
By the equality \eqref{eqn.Pogorelov-first},
for any $i\in\Z_{[1,n]}$,
\begin{gather*}
\frac{3}{4}\frac{u_{11i}^2}{\l_1^2}-2\frac{\vp_i^2}{\vp^2}
=\frac{3}{4}\frac{u_{11i}^2}{\l_1^2}-\frac{1}{2}\frb{\tau u_i \l_i+\frac{u_{11i}}{\l_1}}^2
=\frac{1}{4}\frb{\frac{u_{11i}}{\l_1}-2\tau u_i\l_i}^2
-\frac{3\tau^2}{2}u_i^2\l_i^2
\geq-\frac{3\tau^2K^2}{2}\l_i^2,
\\
\frac{u_{11i}^2}{\l_1^2}
=\frb{2\frac{\vp_i}{\vp}+\tau u_i \l_i}^2
\leq\frac{8\vp_i^2}{\vp^2}+2\tau^2K^2\l_i^2.
\end{gather*}
Moreover, since $\l(D^2u)\in\G_k\subset\G_2$,
\[
\sum_{i=1}^n\l_i^2=(\t u)^2-2\s_2(D^2u)<(\t u)^2\leq n^2\l_1^2.
\]
Therefore
\begin{align*}
\frac{7}{4}F^{11}\sum_{i=1}^n\frac{u_{11i}^2}{\l_1^2}
&\leq\frb{\frac{14|D\vp|^2}{\vp^2}+\frac{7}{2}n^2\tau^2K^2\l_1^2}F^{11}
\leq \frb{\frac{14|D\vp|^2}{(\vp\l_1)^2}+\frac{7}{2}n^2\tau^2K^2}F^{ii}\l_i^2,
\\
\frac{4}{\l_1^2}F^{ii}\l_i^2\frac{u_{11i}^2}{\l_1^2}
&\leq\frb{\frac{32|D\vp|^2}{(\vp\l_1)^2}+8n^2\tau^2K^2}F^{ii}\l_i^2.
\end{align*}
Combining these estimates with \eqref{eqn.Pogorelov-third}
and choosing $\tau:=\frac{1}{(3+23n^2)(1+K^2)}$,
we obtain
\[
0\geq\frb{\tau-\frb{\frac{3+23n^2}{2}}\tau^2K^2
-\frac{46|D\vp|^2}{(\vp\l_1)^2}}
\sum_{i=1}^nF^{ii}\l_i^2
\geq\frb{\frac{\tau}{2}-\frac{46|D\vp|^2}{(\vp\l_1)^2}}
\sum_{i=1}^nF^{ii}\l_i^2.
\]
Consequently, using the gradient bound on $\Om$ in Lemma \ref{lem.comparison-domain},
we obtain
\[
\vp(x_0)\l_1(x_0)\leq\sqrt{\frac{92}{\tau}}|D\vp(x_0)|
\leq C(n,k,K),
\]
and thus
\[
\Phi(x_0,\xi_0)
=\vp^2(x_0)\exp\frb{\frac\tau2|Du(x_0)|^2}\l_1(x_0)
\leq C(n,k,K_0).
\]
Combining this with \eqref{eqn.inner-maximum}
and $\vp=w-u\geq\d$ in $B_{1/2}$ by Lemma \ref{lem.comparison-domain},
we obtain
\begin{equation}\label{eqn.b1/2<b1/4}
\sup_{B_{1/2}}b
\leq\d^{-2}\max_{\ol\Om\times\mathbb S^{n-1}}\Phi
=\d^{-2}\Phi(x_0,\xi_0)
\leq C(n,k,K)\frb{1+\sup_{B_{1/4}}b}
\leq C(n,k,K)\sup_{B_{1/4}}b,
\end{equation}
where the last inequality follows from $1=F(\l)\leq F(b\mathds 1)=b^{k-\ell} F(\mathds1)$.

Finally, for any $y\in B_{1/2}$ and $r\in(0,1/8)$, set
\[
\wt u(x)
:=\frac{u(y+4rx)-u(y)}{(4r)^2}
\quad\text{for any}\ x\in B_2.
\]
Then $D^2\wt u(x)=D^2u(y+4rx)$, $F(D^2\wt u)=1$ in $B_2$,
and $\norm{\wt u}_{C^1(B_2)}\leq K_0/(2r)$.
Applying the inequality \eqref{eqn.b1/2<b1/4} to $\wt u$,
we obtain
\[
\sup_{B_{2r}(y)}b
=\sup_{B_{1/2}}\l_{\max}(D^2\wt u)
\leq
C(n,k,K,r)\sup_{B_{1/4}}\l_{\max}(D^2\wt u)
=C(n,k,K_0,r)\sup_{B_r(y)}b.
\]
This completes the proof.
\end{proof}

\subsection{Proof of Theorem \ref{thm.Hessian-estimate}}\label{sec.main-proof}

The remaining argument is the same as that in \cite[Section  5]{W26}.
For completeness, we give a brief proof.
By rescaling $u(x)$ to $4u(x/2)$,
we prove that $|D^2u(0)|$ is controlled by $\|u\|_{C^1(B_2)}$.
Suppose to the contrary that
there exists a sequence of $k$-convex solutions $u_j\in C^\infty(B_2)$
of equation \eqref{eqn.sigma-quotient}
such that $\norm{u_j}_{C^1(B_2)}\leq A$ and $\abs{D^2u_j(0)}\to+\infty$.
By the Arzel\`a--Ascoli theorem
and the closedness of viscosity solutions,
after passing to a subsequence,
$u_j$ converges uniformly to
a $k$-convex viscosity solution $u$ of $F(D^2u)=1$ in $B_2$.
By the Alexandrov-type differentiability result in \cite[Lemma 5.1]{W26}
(or \cite[Proposition 4.1]{SY25} when $\ell=0$),
there exist $y\in B_{1/8}$ and a quadratic polynomial $P$ such that
\[
u(y+z)=P(y+z)+o(|z|^2),
\quad
\l(D^2P)\in\G_k,
\quad
F(D^2P)=1.
\]
The operator $M\mapsto F(D^2P+M)-1$,
extended as in \cite[Section 5]{W26},
is degenerate elliptic, vanishes at $0$,
and is smooth and uniformly elliptic near $0$.
The second-order expansion of $u$
and the locally uniform convergence of $u_j$
give the required $L^\infty(B_1)$ smallness of
$\rho^{-2}\frb{u_j(y+\rho z)-P(y+\rho z)}$
for a sufficiently small $\rho\in(0,1/8)$
and all sufficiently large $j$.
Savin's small perturbation theorem \cite{Sav07}
therefore gives
\[
\sup_{B_{\rho/2}(y)}\l_{\max}(D^2u_j)\leq C_0
\]
for all sufficiently large $j$.
Finally, choose points $y_0:=y$, $y_1,y_2,\dots,y_N:=0$
on the line segment from $y$ to $0$ such that
$\abs{y_i-y_{i-1}}\leq\rho/2$ for all $i\in\Z_{[1,N]}$, $N\leq1+1/(4\rho)$.
Then $y_i\in B_{1/8}$ and $B_{\rho/2}(y_i)\subset B_\rho(y_{i-1})$.
Repeated application of the doubling inequality (Lemma \ref{lem.doubling})
with centers $y_{i-1}$ and radius $\rho/2$ gives
\begin{align*}
\l_{\max}(D^2u_j(0))
\leq\sup_{B_{\rho/2}(y_N)}\l_{\max}(D^2u_j)
\leq C\frb{n,k,A,\rho}^N\sup_{B_{\rho/2}(y)}\l_{\max}(D^2u_j)
\leq C\frb{n,k,A,\rho}C_0
\end{align*}
for all sufficiently large $j$.
Since $\G_k\subseteq\G_2$ implies
$\abs{D^2u_j}\leq n\l_{\max}(D^2u_j)$,
this contradicts $\abs{D^2u_j(0)}\to+\infty$.
This completes the proof.

\section{Interior Hessian estimates for a class of Hessian equations}\label{sec.general-Hessian}

\begin{proposition}\label{prop.general-Hessian}
Let $n\geq2$,
and let $\G\subset\R^n$ be an open symmetric convex cone with vertex at the origin
satisfying $\G_n\subset\G$.
Assume that $f\in C^2(\G)\cap C^0(\ol\G)$ is a symmetric function
and satisfies the following conditions:
\begin{gather}
f>0\ \text{in}\ \G,
\quad
f=0\ \text{on}\ \p\G,
\label{eqn.nonnegative}
\\
f_i:=\frac{\p f}{\p\l_i}>0\quad\text{in}\ \G,\
\label{eqn.elliptic}
\\
f\ \text{is a concave function},
\label{eqn.sym-concave}
\\
f\ \text{is homogeneous of degree one}.
\label{eqn.homogeneous}
\end{gather}
Suppose also that there exists $c>0$ such that,
for every $\l\in\G$ with $f(\l)=1$,
\begin{gather}
f_i(\l)\geq c\sum_{j=1}^nf_j(\l)
\quad\text{whenever}\ \l_i<0,
\label{eqn.general-negative}
\\
f_i(\l)\frb{1+\l_i^2}\geq c\sum_{j=1}^nf_j(\l)
\quad\text{for all}\ i\in\Z_{[1,n]}.
\label{eqn.general-coefficient}
\end{gather}
If $u\in C^4(B_2)$ is a solution of $F(D^2u):=f(\l(D^2u))=1$ with $\l(D^2u)\in\G$ in $B_2$,
then
\[
\sup_{B_{1/2}}|D^2u|
\leq C\frb{n,f,\G,c,\|u\|_{L^\infty(B_2)}}.
\]
\end{proposition}

\begin{remark}
The assumptions on $\G$
and conditions \eqref{eqn.nonnegative}--\eqref{eqn.sym-concave}
are standard in the Dirichlet solvability theory of Hessian equations and curvature equations;
see \cite{CNS85,Tru90,Tru95}.
Homogeneity \eqref{eqn.homogeneous}
provides the radial growth required for solvability,
gives Euler's identity,
and allows conditions \eqref{eqn.general-negative} and \eqref{eqn.general-coefficient} on $\{f=1\}$
to be extended by rescaling to $\{3/4\leq f\leq1\}$.
The condition \eqref{eqn.general-negative}
is a standard hypothesis used in gradient estimates;
see \cite{L91,SUW04,D20}.
The condition \eqref{eqn.general-coefficient}
supplies the additional control needed for the propagation barrier.
\end{remark}

\begin{remark}
For $f=\s_k^{1/k}$ with $3\leq k\leq n$,
or $f=(\s_k/\s_\ell)^{1/(k-\ell)}$
with $1\leq\ell<k\leq n$ and $k-\ell\geq3$,
all the assumptions of Proposition \ref{prop.general-Hessian} hold on $\G=\G_k$
except for inequality \eqref{eqn.general-coefficient}.
\end{remark}

\begin{proof}[Proof of Proposition \ref{prop.general-Hessian}]
By homogeneity, \eqref{eqn.general-negative} holds throughout $\G$,
and \eqref{eqn.general-coefficient} holds on $3/4\leq f\leq1$
after decreasing $c$.
Furthermore, conditions \eqref{eqn.nonnegative}--\eqref{eqn.homogeneous}
and \eqref{eqn.general-negative}
allow us to apply the interior gradient argument in
\cite[proof of Theorem  1.1]{D20} to $u$.
Thus, by translation and scaling,
it suffices to prove the corresponding estimate in terms of $\norm{u}_{C^1(B_2)}$.

Set $K_0:=1+\norm{u}_{C^1(B_2)}$
and $L_u:=F^{ij}(D^2u)\p_{ij}$,
with $L_v$ defined similarly.
By conditions \eqref{eqn.nonnegative}--\eqref{eqn.homogeneous} and $\G_n\subset\G$,
the solvability theorem \cite[Theorem  1.2]{Tru95} gives an admissible solution
$v\in C^{3,\a}(\ol{B_{3/2}})$, for some $\a\in(0,1)$, of
\[
F(D^2v)=1-\zeta/4\quad\text{in}\ B_{3/2},
\quad v=u\quad\text{on}\ \p B_{3/2},
\]
where $\zeta$ is the fixed cutoff of Section  3.
By convexity and symmetry of $\G$,
together with $\G_n\subset\G$ and $0\in\p\G$,
we have $\G\subset\G_1$.
The comparison principle and $\t v>0$
give $u\leq v$ in $B_{3/2}$.
Applying the same gradient argument to $v$,
we choose $K=K(n,f,\G,c,K_0)\geq1$ such that
\begin{equation}\label{eqn.gradient-F-f}
\norm{Dv}_{L^\infty(B_1)}
+\norm{u}_{C^1(B_2)}
+\norm{\zeta}_{C^1(B_{3/2})}\leq K.
\end{equation}
For admissible solutions of $F(D^2u_j)=1$,
the inclusion $\G\subset\G_1$
gives Alexandrov-type differentiability of locally uniform limits,
as in \cite[proof of Lemma  5.1]{W26}.
Together with the closedness of viscosity solutions,
this yields the seed separation in Lemma  \ref{lem.seed-F}
by the compactness argument in Remark \ref{rem.seed-compactness}.

Next, Euler's identity and concavity imply
\[
\sum_{i=1}^n f_i(\l)\l_i=f(\l),
\quad
\sum_{i=1}^n f_i(\l)\geq f(\mathds1)>0,
\quad
L_v(v-u)\leq-\zeta/4\leq L_u(v-u).
\]
Using these inequalities,
the seed separation,
the inequalities \eqref{eqn.general-coefficient}, \eqref{eqn.general-negative} and $\t v>0$,
we can repeat the propagation and comparison-domain arguments
in Subsections  \ref{subsec.barrier}--\ref{subsec.comparison}.
Furthermore,
the doubling inequality in Lemma  \ref{lem.doubling}
follows from the same calculation,
using symmetry, ellipticity and concavity, $L_u(v-u)\geq0$ outside $\op{supp}\zeta$,
$|D^2u|\leq n\l_{\max}(D^2u)$, and the inequality
\[
1=F(D^2u)
\leq f\frb{\l_{\max}(D^2u)\mathds1}
=\l_{\max}(D^2u)f(\mathds1).
\]

Finally, using the differentiability above,
Savin's theorem \cite[Theorem  1.3]{Sav07},
and the doubling inequality,
we complete the compactness argument
as in the proof of Theorem  \ref{thm.Hessian-estimate}.
\end{proof}

\subsection*{Acknowledgments}

This work was partially supported by NSFC 12171389 and NSFC 11801015.

\renewcommand\refname{References}


\begin{thebibliography}{W26}

\bibitem[C15]{C15}
C.-Q. Chen,
The interior gradient estimate of Hessian quotient equations,
J. Differential Equations, 259 (2015), 1014--1023.

\bibitem[CHO16]{CHO16}
C.-Q. Chen, F. Han, and Q.-Z. Ou,
The interior $C^2$ estimate for the Monge-Amp\`ere equation in dimension $n=2$,
Anal. PDE, 9 (2016), 1419--1432.


\bibitem[CNS85]{CNS85}
L. Caffarelli, L. Nirenberg, and J. Spruck,
The Dirichlet problem for nonlinear second-order elliptic equations.
III. Functions of the eigenvalues of the Hessian,
Acta Math., 155 (1985), 261--301.

\bibitem[CT05]{CT05}
N. Chaudhuri and N. S. Trudinger,
An Alexandrov type theorem for $k$-convex functions,
Bull. Austral. Math. Soc., 71 (2005), 305--314.

\bibitem[CW01]{CW01}
K.-S. Chou and X.-J. Wang,
A variational theory of the Hessian equation,
Comm. Pure Appl. Math., 54 no. 9 (2001), 1029--1064.

\bibitem[CWY09]{CWY09}
J.-Y. Chen, M. Warren, and Y. Yuan,
A priori estimate for convex solutions to special Lagrangian equations and its application,
Comm. Pure Appl. Math., 62 (2009), 583--595.

\bibitem[CY10]{CY10}
S.-Y. A. Chang and Y. Yuan,
A Liouville problem for the Sigma-2 equation,
Discrete Contin. Dyn. Syst., 28 no. 2 (2010), 659--664.

\bibitem[CZZ26]{CZZ26}
R.-S. Chen, X.-C. Zhou, and R.-X. Zhu,
Interior $C^{2,\a}$ regularity for the quadratic Hessian equation,
arXiv:2608.29484 (2026).


\bibitem[D20]{D20}
W.-S. Dong,
Interior and boundary gradient estimates for solutions to Hessian equations
satisfying Neumann boundary conditions,
Nonlinear Anal., 195 (2020), 111733.

\bibitem[DZ26]{DZ26}
W.-S. Dong and R.-J. Zhang,
Interior Hessian estimates for Hessian quotient equations,
arXiv:2608.19087v2 (2026).

\bibitem[Fan26]{Fan26}
Z.-Y. Fan,
Hessian estimates for the sigma-$2$ equation
with variable right-hand side terms in dimension $4$,
Adv. Math., 494 (2026), Paper No. 110953.

\bibitem[F26a]{F26a}
C. Y. Fung,
Doubling argument of the Hessian estimate for the Hessian quotient equations,
arXiv:2607.21982 (2026).


\bibitem[GQ19]{GQ19}
P.-F. Guan and G.-H. Qiu,
Interior $C^2$ regularity of convex solutions
to prescribing scalar curvature equations,
Duke Math. J., 168 (2019), 1641--1663.

\bibitem[GS26]{GS26}
P.-F. Guan and M. Sroka,
A special concavity property for positive Hessian quotient operators,
Discrete Contin. Dyn. Syst., 54 (2026), 50--60.

\bibitem[H59]{H59}
E. Heinz,
On elliptic Monge--Amp\`ere equations and Weyl's embedding problem,
J. Analyse Math., 7 (1959), 1--52.

\bibitem[HM26]{HM26}
F. Han, Y. Ma,
Interior curvature estimates for graphical curvature quotient equations,
arXiv:2609.39276 (2026).


\bibitem[JS26]{JS26}
H.-M. Jiao and Z.-A. Sui,
Interior Hessian estimates for Hessian quotient equations in dimension three,
arXiv:2602.14064 (2026).

\bibitem[K87]{K87}
N. J. Korevaar,
A priori interior gradient bounds for solutions to elliptic Weingarten equations,
Ann. Inst. H. Poincar\'e Anal. Non Lin\'eaire, 4 no. 5 (1987), 405--421.

\bibitem[L91]{L91}
Y.-Y. Li,
Interior gradient estimates for solutions of certain fully nonlinear elliptic equations,
J. Differential Equations, 90 (1991), 172--185.

\bibitem[Liu21]{Liu21}
J.-K. Liu,
Interior $C^2$ estimate for Monge--Amp\`ere equation in dimension two,
Proc. Amer. Math. Soc., 149 (2021), 2479--2486.


\bibitem[Lu23]{Lu23}
S.-Y. Lu,
Interior $C^2$ estimate for Hessian quotient equation in dimension three,
arXiv:2311.05835 (2023).

\bibitem[Lu25]{Lu25}
S.-Y. Lu,
Interior $C^2$ estimate for Hessian quotient equation in general dimension,
Ann. PDE, 11 (2025), Paper No. 17, 26 pp.


\bibitem[LT26a]{LT26a}
S.-Y. Lu and Y.-L. Tsai,
A note on interior $C^2$ estimate for general Hessian quotient equation,
Commun. Pure Appl. Anal., 33 (2026), 88--100.

\bibitem[LT26b]{LT26b}
S.-Y. Lu and Y.-L. Tsai,
Pogorelov type interior $C^2$ estimate for Hessian quotient equation and its application,
J. reine angew. Math., 831 (2026), 155--184.

\bibitem[LW26a]{LW26a}
Z.-S. Li and K. Wu,
A concavity inequality and interior $C^2$ estimate for Hessian quotient equations,
arXiv:2608.17405 (2026).

\bibitem[LW26b]{LW26b}
Z.-S. Li and K. Wu,
Interior Hessian estimates for the quadratic Hessian equation,
arXiv:2608.23233v2 (2026).

\bibitem[M21]{M21}
C. Mooney,
Strict $2$-convexity of convex solutions
to the quadratic Hessian equation,
Proc. Amer. Math. Soc., 149 (2021), 2473--2477.

\bibitem[M25]{M25}
C. Mooney,
Remarks on the quadratic Hessian equation,
arXiv:2505.14586 (2025).

\bibitem[MSY19]{MSY19}
M. McGonagle, C. Song, and Y. Yuan,
Hessian estimates for convex solutions to quadratic Hessian equation,
Ann. Inst. H. Poincar\'e C Anal. Non Lin\'eaire,
36 (2019), 451--454.

\bibitem[MY26]{MY26}
X.-Q. Mei and J. Yan,
Interior $C^2$ estimate for semi-convex solutions to a class of
Hessian quotient equations in arbitrary dimensions,
arXiv:2604.23349 (2026).

\bibitem[P78]{P78}
A. V. Pogorelov,
The Minkowski multidimensional problem,
Halsted Press, New York--Toronto--London, 1978.

\bibitem[Q24]{Q24}
G.-H. Qiu,
Interior Hessian estimates for $\s_2$ equations in dimension three,
Front. Math., 19 (2024), 577--598.

\bibitem[QY26a]{QY26a}
G.-H. Qiu and J. Yan,
Interior curvature estimates for the graphical scalar curvature equation in all dimensions,
arXiv:2609.02581v2 (2026).

\bibitem[QY26b]{QY26b}
G.-H. Qiu and J. Yan,
Interior estimates for Hessian quotient equations,
arXiv:2609.24549 (2026).


\bibitem[Sav07]{Sav07}
O. Savin,
Small perturbation solutions for elliptic equations,
Comm. Partial Differential Equations, 32 (2007), 557--578.

\bibitem[S26]{S26}
R. Shankar,
Hessian estimates for the special Lagrangian equation by doubling,
Anal. PDE, 19 (2026), 339--352.

\bibitem[SUW04]{SUW04}
W.-M. Sheng, J. Urbas, and X.-J. Wang,
Interior curvature bounds for a class of curvature equations,
Duke Math. J., 123 no. 2 (2004), 235--264.

\bibitem[SY20]{SY20}
R. Shankar and Y. Yuan,
Hessian estimate for semiconvex solutions to the sigma-$2$ equation,
Calc. Var. Partial Differential Equations,
59 (2020), Paper No. 30, 12 pp.


\bibitem[SY25]{SY25}
R. Shankar and Y. Yuan,
Hessian estimates for the sigma-$2$ equation in dimension four,
Ann. of Math., 201 (2025), 489--513.

\bibitem[T26]{T26}
Y.-L. Tsai,
A concavity inequality for Hessian quotient equations,
arXiv:2608.16383 (2026).

\bibitem[Tru90]{Tru90}
N. S. Trudinger,
The Dirichlet problem for the prescribed curvature equations,
Arch. Rational Mech. Anal., 111 (1990), 153--179.

\bibitem[Tru95]{Tru95}
N. S. Trudinger,
On the Dirichlet problem for Hessian equations,
Acta Math., 175 (1995), 151--164.


\bibitem[U90]{U90}
J. I. E. Urbas,
On the existence of nonclassical solutions
for two classes of fully nonlinear elliptic equations,
Indiana Univ. Math. J., 39 (1990), 355--382.

\bibitem[WY09]{WY09}
M. Warren and Y. Yuan,
Hessian estimates for the sigma-$2$ equation in dimension $3$,
Comm. Pure Appl. Math., 62 (2009), 305--321.

\bibitem[WY14]{WY14}
D.-K. Wang and Y. Yuan,
Hessian estimates for special Lagrangian equations
with critical and supercritical phases in general dimensions,
Amer. J. Math., 136 (2014), 481--499.

\bibitem[W26]{W26}
K. Wu,
Interior Hessian estimates for $k$-convex solutions
of Hessian quotient equations with $2k>n$,
arXiv:2609.23006v3 (2026).



\bibitem[Z24]{Z24}
X.-C. Zhou,
Notes on generalized special Lagrangian equation,
Calc. Var. Partial Differential Equations, 63 (2024), Paper No. 197, 28 pp.

\end{thebibliography}
\end{document}